\documentclass[11pt]{article}
\usepackage{latexsym}
\usepackage{epsfig,enumerate}
\usepackage{amsmath,amsthm,amssymb}
\usepackage{algorithm2e}
\usepackage{xcolor}
\usepackage{hyperref}

\usepackage{graphicx}
\numberwithin{equation}{section}
\definecolor{brown}{cmyk}{0, 0.72, 1, 0.45}
\definecolor{grey}{gray}{0.5}

\renewcommand{\epsilon}{\varepsilon}

\def\bh{\vec{b}}
\everymath{\displaystyle}

\newcounter{rot}%\addtocounter{rot}{1}, \therot

     \def\l{\ell}

\newtheorem{maintheorem}{Theorem}

\newtheorem*{conjecture*}{Conjecture}
\newtheorem{theorem}{Theorem}[section]
\newtheorem*{theorem*}{Theorem}

\newtheorem{fact}[theorem]{Fact}

\newcommand{\rbrac}[1]{\left(#1\right)}
\newcommand{\cbrac}[1]{\left\{#1\right\}}
\newcommand{\sbrac}[1]{\left[{ #1}\right]}

\def\E{\mathbb{E}}

\newcommand{\ignore}[1]{}

\newcommand{\beq}[1]{\begin{equation}\label{#1}}
\newcommand{\eeq}{\end{equation}}

\newcommand{\mc}[1]{\mathcal{#1}}

\title{The set of ratios of derangements to permutations in digraphs is dense in $[0, 1/2]$.}

\author{Bethany Austhof \thanks{Department of Mathematics, Statistics, and Computer Science, University of Illinois Chicago, Chicago, IL, USA \texttt{bausth2@uic.edu}}\and  Patrick Bennett\thanks{Department of Mathematics, Western Michigan University, Kalamazoo, MI, USA
\texttt{patrick.bennett@wmich.edu}. Research supported in part by Simons Foundation Grant \#426894.}  \and Nick Christo \thanks{ Department of Mathematics, Statistics, and Computer Science, University of Illinois Chicago, Chicago, IL, USA \texttt{nchrist5@uic.edu}}}
\date{}
\begin{document}
\maketitle

\begin{abstract}
A {\em permutation} in a digraph $G=(V, E)$ is a bijection $f:V \rightarrow V$ such that for all $v \in V$ we either have that $f$ fixes $v$ or $(v, f(v)) \in E$. A {\em derangement} in $G$ is a permutation that does not fix any vertex. In \cite{BDHHL} it is proved that in any digraph, the ratio of derangements to permutations is at most $1/2$. Answering a question posed in \cite{BDHHL}, we show that the set of possible ratios of derangements to permutations in digraphs is dense in the interval $[0, 1/2]$.
\end{abstract}

\section{Introduction}

A {\em permutation} in a digraph $G=(V, E)$ is a bijection $f:V \rightarrow V$ such that for all $v \in V$ we either have that $f$ fixes $v$ or $(v, f(v)) \in E$. A {\em derangement} in $G$ is a permutation that does not fix any vertex. We define the parameter $(d/p)_G$ to be the ratio of derangements to permutations in $G$. 

Bucic, Devlin, Hendon, Horne and Lund \cite{BDHHL} showed that $(d/p)_G \le 1/2$ for all digraphs $G$, with equality if and only if $G$ is a directed cycle. They also gave a construction (the blow-up of a directed cycle) that can achieve a ratio arbitrarily close to but not equal to $1/2$. Let $S=\{(d/p)_G\; :\; G \;\;\text{is a digraph}\}$ be a set of values arising as a ratio $(d/p).$ In \cite{BDHHL} they analyzed the ratio $(d/p)_G$ for the random graph $G=G(n, m)$, and as a corollary of this analysis they showed that $S$ is dense in $[0, 1/e]$. This corollary follows from two facts:  $(d/p)_G$ is concentrated around its mean, and by choosing a suitable value of $m$ one can make the expected ratio $(d/p)_G$ close to any given value in $[0, 1/e]$. At the end of the paper \cite{BDHHL} they ask whether $S$ is dense in $[0, 1/2]$. Our main theorem, below, answers this question in the positive. 
 
\begin{maintheorem}\label{thm:main}
The set of possible ratios of derangements to permutations in digraphs is dense in $[0, 1/2]$.
\end{maintheorem}

The construction we use, described in more detail later, is a random subgraph of the blow-up of a directed cycle. The main part of the proof is an application of the second moment method (see \cite{FK}, for example, for an introduction to the method) to show that the number of derangements and permutations are concentrated around their expectations. 

\section{Proof of Theorem \ref{thm:main}}

\subsection{Outline}
First we outline the proof. Suppose we are given a fixed real number  $r \in [0, 1/2]$. We will show that there exists a sequence of digraphs $G_k$ such that the ratio of derangements to permutations in $G_k$ is $r+o(1)$ as $k \rightarrow \infty$ (which proves Theorem \ref{thm:main}). If $r=0$ or $1/2$ this is trivial. Indeed, for $r=0$ observe that a digraph with one vertex and no edges has no derangements and one permutation, and for $r=1/2$ observe that any directed cycle has one derangement and two permutations. So we assume $0 < r < 1/2$.

Our construction is as follows. As defined in \cite{BDHHL}, let the digraph $D_{k,\l}$ where $k\geq 1$ and $\l\geq 2$ have vertices $v_{ij}$ for $i\in [k]$ and $j\in[\l]$ such that $(v_{ij}, v_{lm})\in E(D_{k,\l})$ if and only if $m=j+1\, \text{mod}\, l$. In other words, $D_{k, \l}$ is the blow-up of a directed $\l$-cycle where each vertex is expanded to a set of $k$ vertices. We let $V_i = \{v_{ij}: j \in [\l]\}$. 
 As was shown in \cite{BDHHL}, the number of derangements on $D_{k,\l}$ is $(k!)^\l$ and the number of permutations on $D_{k,\l}$ is $\sum_{i=0}^k\left(\binom{k}{i} (k-i)!\right)^\l$. Hence, $(d/p)_{k,\l}=\left(\sum_{i=0}^k \left(\frac{1}{i!} \right)^\l\right)^{-1}$  can be made arbitrarily close to 1/2 by choosing $\ell$ large enough (even for large $k$). This construction yields a graph for which the ratio of derangements to permutations is arbitrarily close to 1/2 but not exactly 1/2. We will also use this construction, but we will randomly remove some edges. By taking a random subraph  we can ``interpolate" between $D_{k, \l}$ (a dense digraph whose ratio of derangements to permutations is close to $1/2$) and a sparse random digraph (whose ratio is $0$). 
 
 In this paper all asymptotics are as $k \rightarrow \infty$. $\l$ is treated as fixed. We use standard big-O, little-o and $\Omega$ notation. We write $x \sim y$ if $x=(1+o(1))y$. All logarithms are base $e$. 
 
 \subsection{Proof details}
 
 Let the random graph $G_{k,\l}(m)$  be chosen uniformly from among all subgraphs of $D_{k,\l}$ with $m$ edges. We will fix some $p, \l$ and let $m = pk^2 \l$ (so $p$ is the probability that any particular edge of $D_{k,\l}$ becomes an edge of $G_{k,\l}$). Let the random variables $X, Y$ be the number of derangements and permutations in $G_{k,\l}(m)$ respectively. Let $\mc{D}, \mc{P}$ be the collection of all possible derangements and permutations on $D_{k,\l}(m)$. 
 
 \subsubsection{First moments of $X, Y$}
 We have
 \begin{align}
     \mathbb{E}[X] &= \sum_{D \in \mc{D}}\mathbb{P}[D\subseteq G_{k, \l}]=(k!)^\l\frac{\binom{k^2\l-k \l}{m-k \l} }{\binom{k^2\l}{m}}\nonumber\\
     &=(k!)^\l \left(\frac{m}{k^2\ell} \right)^{k\ell}\exp\left\{\frac{k^2 \ell^2}{2}\left(\frac{1}{k^2\ell}-\frac{1}{m}\right)+O\left(\frac{k^3}{m^2} + \frac{k}{m} \right) \right\}\nonumber\\
     &\sim (k!)^\l p^{k\ell}\exp\left\{\frac{ \ell}{2}\left(1-\frac{1}{p}\right) \right\}\label{eqn:EX}
 \end{align} 
where on the second line we have used the following fact:
\begin{fact}\label{fact:edgeprob}
\[
\frac{\binom{a-x}{b-x} }{\binom{a}{b}} = \frac{(b)_x}{(a)_x} = \rbrac{\frac ba}^x \exp\cbrac{\frac{x^2}{2}\rbrac{\frac1a - \frac1b} + O\rbrac{\frac{x^3}{b^2} + \frac xb}}.
\]
\end{fact}
For completeness we include the proof although it is well-known. 
\begin{proof}
\begin{align*}
    \frac{(b)_x}{(a)_x} &= \rbrac{\frac ba}^x \cdot \frac{1 \rbrac{1-\frac 1b}\rbrac{1-\frac 2b} \cdots \rbrac{1-\frac {x-1}b} }{1 \rbrac{1-\frac 1a}\rbrac{1-\frac 2a} \cdots \rbrac{1-\frac {x-1}a}}\\
    &= \rbrac{\frac ba}^x \cdot \exp\cbrac{\sum_{i=0}^{x-1} \sbrac{ \ln\rbrac{1-\frac{i}{b} }- \ln\rbrac{1-\frac{i}{a}} }}\\
    &= \rbrac{\frac ba}^x \cdot \exp\cbrac{\sum_{i=0}^{x-1} \sbrac{ -\frac ib + \frac ia + O\rbrac{\frac{i^2}{a^2} + \frac{i^2}{b^2}} }}\\
    &= \rbrac{\frac ba}^x \cdot \exp\cbrac{ \frac{x(x-1)}{2} \rbrac{\frac1a - \frac1b} + O\rbrac{\frac{x^3}{b^2}}  }\\
    &= \rbrac{\frac ba}^x \exp\cbrac{\frac{x^2}{2}\rbrac{\frac1a - \frac1b} + O\rbrac{\frac{x^3}{b^2} + \frac xb}}.
\end{align*}
\end{proof}
Before we calculate $\E[Y]$ we introduce a function $f_\l(x)$. For any integer $\l \ge 1$, let
\begin{equation}\label{eqn:deff}
    f_\l(x) := \sum_{i=0}^\infty\frac{ x^{i \l}}{(i!)^\l} . 
\end{equation}
 Note that the above power series for $f_\l(x)$ converges for all $x$ and therefore in particular each $f_\l$ is continuous in $x$. We have
 \begin{align}
     \mathbb{E}[Y] &= \sum_{P \in \mc{P}}\mathbb{P}[P\subseteq G_{k, \l}]=\sum_{i=0}^k\left( \binom{k}{i}(k-i)!\right)^\l \frac{\binom{k^2\l-(k-i) \l}{m-(k-i) \l} }{\binom{k^2\l}{m} }\nonumber\\
     &= \sum_{i=0}^k\left( \frac{k!}{i!}\right)^\l  \left(\frac{m}{k^2\ell} \right)^{(k-i) \l} \exp\left\{\frac{(k-i)^2 \l^2}{2}\left(\frac{1}{k^2\ell}-\frac{1}{m}\right)+O\left(\frac{k^3 }{m^2}  + \frac km \right) \right\}\nonumber\\
       &= (k!)^\l p^{k\l} \sum_{i=0}^k\left( \frac{1}{i!}\right)^\l  p^{-i \l} \exp\left\{\frac{ \l}{2}\left(1-\frac{1}{p}\right) + O\rbrac{\frac{i+1}{k}} \right\} \label{eqn:EY1}.\end{align} 
       We split the above sum into two ranges of $i
       $.  Note that for $0 \le i \le \sqrt{k}$ we have \newline $\exp\cbrac{O\rbrac{\frac{i+1}{k}}} = 1+O\rbrac{\frac{1}{\sqrt{k}}}$, while for   $\sqrt{k} \le i \le k$ we have $\exp\cbrac{O\rbrac{\frac{i+1}{k}}}=O(1)$. Thus line \eqref{eqn:EY1} becomes 
    \begin{equation}
        (k!)^\l p^{k\l} \sbrac{ \rbrac{1+ O\rbrac{\frac{1}{\sqrt{k}}}}\exp\left\{\frac{ \l}{2}\left(1-\frac{1}{p}\right)  \right\}\sum_{0 \le i \le \sqrt{k}} \left( \frac{1}{i!}\right)^\l  p^{-i \l}  + O(1)\sum_{\sqrt{k}<i \le k} \left( \frac{1}{i!}\right)^\l  p^{-i \l} } \label{eqn:EY2}.
    \end{equation}
    As $k \rightarrow \infty$ we have $$\sum_{0 \le i \le \sqrt{k}} \left( \frac{1}{i!}\right)^\l  p^{-i \l} \rightarrow f_\l(1/p), $$ and 
    $$\sum_{\sqrt{k} < i \le k} \left( \frac{1}{i!}\right)^\l  p^{-i \l} \le \sum_{i=\sqrt{k}}^{\infty} \left( \frac{1}{i!}\right)^\l  p^{-i \l} =o(1) $$ since the latter is the tail of a convergent series. Thus, returning to our estimate of $\mathbb{E}[Y]$ on line \eqref{eqn:EY2}, we have
      \begin{equation}
         \mathbb{E}[Y] \sim (k!)^\l p^{k\l} \exp\left\{\frac{ \l}{2}\left(1-\frac{1}{p}\right)  \right\}f_\l(1/p). \label{eqn:EY}
      \end{equation}
      
      \subsubsection{Choosing $p, \l$}
      
Now that we know $\E[X], \E[Y]$ we will choose $p, \l$ to make sure that the ratio of $\E[X]$ to $\E[Y]$ is close to $r$. Using lines \eqref{eqn:EX} and \eqref{eqn:EY} we have
 \[
 \frac{\E[X]}{\mathbb{E}[Y]} \sim \frac{1}{f_\l\rbrac{\frac1p}},
 \]
 so we would like to choose $\l$ and $0<p<1$ so that $f_\l(1/p) = 1/r$. We have 
\[
\lim_{x \rightarrow \infty} f_\l(x) = \infty , \qquad f_\l(1) = \sum_{i=0}^k\left( \frac{1}{i!}\right)^\l = 1 + 1 + \frac{1}{2^\l} +  \frac{1}{6^\l} + \frac{1}{24^\l} + \ldots.
\]
Note that we can make $f_\l(1)$ arbitrarily close to 2 by taking $\l$ large. Indeed, we have $f_\l(1) \ge 2$ and 
\begin{align*}
     f_\l(1) = 2+ \sum_{i\ge 2}\left( \frac{1}{i!}\right)^\l &\leq 2+ \sum_{i\geq 2} \left( \frac{1}{2^{i-1} }\right)^\ell =2+ \frac{1}{2^\ell - 1}.
    \end{align*}
Since $r<1/2$, we can choose $\l$ so that $f_{\l}(1) <  1/r$. Then by the intermediate value theorem there is some $x \in (1, \infty)$ such that $f_\l(x) = 1/r$. We choose $p$ to be the value $1/x$, so $0<p<1$  and $f_\l(1/p) = 1/r$. So we view $\l$ and $p$ as constants determined entirely by $r$.

\subsubsection{Second moments of $X, Y$}

In this section we show that $\E[X^2] \sim \E[X]^2$ and $\E[Y^2] \sim \E[Y]^2$. This will complete the proof, since then by the second moment method we have that 
\[
\frac XY \sim \frac{\E[X]}{\E[Y]} \sim \frac{1}{f_\l(1/p)} = r. 
\]
with probability approaching 1 as $k$ goes to infinity. 

To help us estimate $\E[X^2], \E[Y^2]$ we will find the function $h(a, b)$ (defined below) useful. Suppose we have some fixed matching $B$ of $b$ many edges in the graph $K_{a, a}$. Then by inclusion-exclusion the number of perfect matchings that do not have any edges from $B$ is 
\[
h(a, b) := \sum_{w=0}^b (-1)^w \binom{b}{w} (a-w)!.
\]

Note that we always have $h(a, b) \le a!$. We will now observe that, roughly speaking, $h(a,b)\approx \frac{a!}{e}$ whenever $b \approx a \rightarrow \infty$. More formally we have the following

\begin{fact}\label{fact:hest}
    Suppose $a-a^{1/10} \le b \le a$. Then we have 
    \[
    h(a, b) = \rbrac{1 + O(a^{-4/5})} \frac{a!}{e}
    \]
    as $a \rightarrow \infty$
\end{fact} 
\begin{proof}
We have \begin{align}\label{eqn:h}
     h(a,b)=\sum_{0 \le w \le b} (-1)^w \binom{b}{w} (a-w)!= a! \sum_{0 \le w \le b}  \frac{(-1)^w}{w!}\frac{(b)_w}{(a)_w}.
 \end{align}
Now, for $0 \le w \le a^{1/10}$ we have by Fact \ref{fact:edgeprob} that
\begin{align*}
    \frac{(b)_w}{(a)_w} &= \rbrac{\frac ba}^w \exp\cbrac{ \frac{w^2}{2} \rbrac{\frac1a - \frac1b} + O\rbrac{\frac{w^3}{b^2} + \frac wb}}\\
    &= \rbrac{1 + O\rbrac{a^{-9/10}}}^{O\rbrac{a^{1/10}}} \exp \cbrac{O(a^{-4/5}} = 1 + O(a^{-4/5}).
\end{align*}
Meanwhile for $w \ge a^{1/10}$ we have that the corresponding term in line \eqref{eqn:h} has absolute value
\begin{align*}
    \frac{1}{w!}\frac{(b)_w}{(a)_w} \le \frac{1}{(a^{1/10})!} = \exp\cbrac{ - \Omega\rbrac{a^{1/10} \log a}}
\end{align*}
by Stirling's approximation. Thus, the sum of all such terms in line \eqref{eqn:h} is at most 
\[
b \exp\cbrac{ - \Omega\rbrac{a^{1/10} \log a}} = O(a^{-4/5})
\]
(this bound is quite comfortable). By the Alternating Series Test we have that 
\[
\sum_{0 \le w \le a^{1/10}} \frac{(-1)^w}{w!} = \frac 1e + O\rbrac{\frac{1}{(a^{1/10})!}}  = \frac 1e + O(a^{-4/5}).
\]
 Breaking up the sum for $h(a, b)$ we have \begin{align*}
     h(a, b) &= a! \sbrac{\sum_{0 \le w \le a^{1/10}}  \frac{(-1)^w}{w!}\frac{(b)_w}{(a)_w}  + \sum_{a^{1/10}< w \le b}  \frac{(-1)^w}{w!}\frac{(b)_w}{(a)_w}} \\
     &= a! \sbrac{\rbrac{1 + O(a^{-4/5})} \sum_{0 \le w \le a^{1/10}} \frac{(-1)^w}{w!}  +  O(a^{-4/5}) } \\
     &= \rbrac{1 + O(a^{-4/5})} \frac{a!}{e}.
 \end{align*}
\end{proof}

 We find that
 \begin{align}
     \mathbb{E}[X^2] &= \sum_{D,D' \in \mc{D}}\mathbb{P}[D,D'\subseteq G_{k \l}]=(k!)^\l\sum_{D' \in \mc{D}}\mathbb{P}[D_0,D'\subseteq M]\nonumber\\
     &=(k!)^\l\sum_{b=0}^{k\l}\left[\frac{\binom{k^2\l-(2k \l-b)}{m-(2k \l-b)} }{\binom{k^2\l}{m} }\sum_{\bh \in S_b}\prod_{c=1}^\l\binom{k}{b_c}h(k-b_c, k-b_c) \right].\label{eqn:EX2}
 \end{align}
 where in the inner sum, $S_b$ is the set of $\l$-dimensional vectors $\bh = (b_1, \ldots b_\l)$ whose components are nonnegative integers summing to $b$. 
 
 By \ref{fact:edgeprob}, if $b \le k^{1/10}$ then we have 
\begin{align*}
   \frac{\binom{k^2\l-(2k \l-b)}{m-(2k \l-b)} }{\binom{k^2\l}{m} } &= p^{2k\ell-b}\exp\left\{\frac{(2k\l-b)^2}{2}\left(\frac{1}{k^2\ell}-\frac{1}{m}\right)+O\left(\frac{k^3}{m^2}+\frac{k}{m}\right) \right\}  \\
   & = \rbrac{1 + O(k^{-9/10})} p^{2k\ell-b}\exp\left\{2\ell \left(1-\frac{1}{p}\right) \right\} 
\end{align*}
and by Fact \ref{fact:hest} we have
$ h(k-b_c, k-b_c) =\rbrac{1 + O(k^{-4/5})} \frac{(k-b_c)!}{e}$. Therefore the term corresponding to $b$ in \eqref{eqn:EX2} is 
 \begin{align*}
      \frac{\binom{k^2\l-(2k \l-b)}{m-(2k \l-b)} }{\binom{k^2\l}{m} }\sum_{\bh \in S_b}\prod_{c=1}^\l & \binom{k}{b_c}h(k-b_c, k-b_c)\\
      &= \rbrac{1 + O(k^{-4/5})}p^{2k\ell-b}\exp\left\{2\ell \left(1-\frac{1}{p}\right) \right\} \sum_{\bh\in S_b}\prod_{c=1}^\l\binom{k}{b_c} \frac{(k-b_c)!}{e}  \\
      &= \rbrac{1 + O(k^{-4/5})}(k!)^{\l} p^{2k\ell-b}\exp\left\{2\ell \left(1-\frac{1}{p}\right) - \l \right\} \sum_{\bh\in S_b} \prod_{c=1}^\l \frac{1}{b_c!} \\ 
      &=\rbrac{1 + O(k^{-4/5})}(k!)^{\l} p^{2k\ell-b}\exp\left\{\ell \left(1-\frac{2}{p}\right)  \right\}  \frac{\l^b}{b!}
 \end{align*}
 where on the last line we used the multinomial formula. Meanwhile if $b \ge k^{1/10}$ then the term corresponding to $b$ in \eqref{eqn:EX2} is 
  \begin{align*}
      \frac{\binom{k^2\l-(2k \l-b)}{m-(2k \l-b)} }{\binom{k^2\l}{m} }\sum_{\bh \in S_b}\prod_{c=1}^\l & \binom{k}{b_c}h(k-b_c, k-b_c) \le   \sum_{\bh\in S_b}\prod_{c=1}^\l\binom{k}{b_c} (k-b_c)!\\
      & = (k!)^\l  \frac{\l^b}{b!} = (k!)^\l \cdot \exp\cbrac{-\Omega\rbrac{k^{1/10} \log k}}
 \end{align*}
 and so the sum of all terms in \eqref{eqn:EX2} with $b \ge k^{1/10}$ is at most 
 \[
 k\l \cdot (k!)^\l \cdot \exp\cbrac{-\Omega\rbrac{k^{1/10} \log k}} = (k!)^\l \cdot O\rbrac{k^{-4/5}}.
 \]
 Therefore 
  \begin{align*}
     \mathbb{E}[X^2] &=(k!)^\l\sum_{b=0}^{k\l}\left[\frac{\binom{k^2\l-(2k \l-b)}{m-(2k \l-b)} }{\binom{k^2\l}{m} }\sum_{\bh \in S_b}\prod_{c=1}^\l\binom{k}{b_c}h(k-b_c, k-b_c) \right]\\
     &= (k!)^\l \sbrac{\sum_{0 \le b \le k^{1/10} } \rbrac{1 + O(k^{-4/5})}(k!)^{\l} p^{2k\ell-b}\exp\left\{\ell \left(1-\frac{2}{p}\right)  \right\}  \frac{\l^b}{b!} + (k!)^\l \cdot   O\rbrac{k^{-4/5}}}\\
     & = \rbrac{1 + O(k^{-4/5})} (k!)^{2\l}p^{2k\l}\exp\left\{\ell \left(1-\frac{2}{p}\right)  \right\} \sum_{0 \le b \le k^{1/10} }  p^{-b}  \frac{\l^b}{b!}\\
     & = \rbrac{1 + O(k^{-4/5})} (k!)^{2\l}p^{2k\l}\exp\left\{\ell \left(1-\frac{2}{p}\right)  \right\}\cdot \rbrac{\exp\left\{\frac{\l}{p}  \right\} + O(k^{-4/5}) } \\
      & = \rbrac{1 + O(k^{-4/5})} (k!)^{2\l}p^{2k\l}\exp\left\{\ell \left(1-\frac{1}{p}\right)  \right\} \\
      & \sim \mathbb{E}[X]^2
 \end{align*}

\vspace{1cm}
 For $\E[Y^2]$ we find an exact expression to be cumbersome, but the following upper bound will suffice: 
 \begin{align}
     \mathbb{E}[Y^2] &\le  \sum_{\substack{0 \le i, j \le k\\ 0 \le b \le k \l \\ \bh \in S_b}}  \frac{\binom{k^2\ell-(2k\ell-(i+j)\ell-b)}{m-(2k\ell-(i+j)\ell-b)} }{\binom{k^2\ell}{m} } \left(\frac{k!}{i!}\right)^\ell \prod_{c=1}^\ell \binom{k-i}{b_c}\binom{k-b_c}{j} h(k-j-b_c, k-i-b_c-2j) \label{eqn:EYsquared}
    %   &= (1+o(1))\left( \sum_{i=0}^k \left(\frac{k!}{i!}\right)^\ell\right) \sum_{j=0}^k\sum_{(b_1,\dots, b_\ell)}R \prod_{t=1}^\ell \frac{(k-i)!}{b_t!j!(k-i-j-b_t)!}\frac{(k-j-b_t)!}{e} \\
    %   &= (1+o(1))\left( \sum_{i=0}^k \left(\frac{k!}{i!}\right)^\ell\right) \sum_{j=0}^k\sum_{(b_1,\dots, b_\ell)}\frac{R}{(j!\,e)^\ell}\prod_{t=1}^\ell \frac{ S}{b_t!}
 \end{align}
The term corresponding to a tuple $(i, j, b, \bh)$ above is an upper bound on the contribution to $\mathbb{E}[Y^2]$ due to pairs of permutations $(P, P')$ such that $P$ fixes $i$ vertices per part, $P'$ fixes $j$ vertices per part, and $P$ and $P'$ share a total of $b$ edges where $b_c$ of the shared edges are between $V_c$ and part $V_{c+1}$. The first factor is the edge probability, and the next factor is the number of choices for $P$. The next factor is an upper bound on the number of choices for $P'$. Indeed, we choose the edges of $P'$ from $V_c$ to  $V_{c+1}$ by first choosing $b_c$ edges of $P$ to be shared, then we choose $j$ vertices in $V_c$ to be fixed by $P'$, and finally we choose a matching between the remaining vertices (the vertices of $V_c \cup V_{c+1}$ that are not fixed by $P'$ and are not endpoints of the $b_c$ shared edges already chosen). This matching must avoid any edges of $P$, and the vertices to be matched induce at least $k-i-b_c-2j$ edges of $P$, explaining the last factor above. 

We will now estimate the significant terms in \eqref{eqn:EYsquared}. Assume $i, j, b \le k^{1/10}$. Then by Fact
\ref{fact:edgeprob} \begin{align*}
     \frac{\binom{k^2\ell-(2k\ell-(i+j)\ell-b)}{m-(2k\ell-(i+j)\ell-b)} }{\binom{k^2\ell}{m} } &= p^{2k\ell-(i+j)\ell-b}\exp\left\{\frac{(2k\ell-(i+j)\ell-b)^2}{2k^2\l}\rbrac{1-\frac{1}{p}} +O\rbrac{\frac{1}{k}}\right\} \\
     &= p^{2k\ell-(i+j)\ell-b}\exp\left\{2\ell\left(1-\frac{1}{p}\right) +O\rbrac{k^{-9/10}}\right\}.
 \end{align*}
 Next we estimate 
 \begin{align*}
h(k-j-b_c, k-i-b_c-2j) = \rbrac{1+O\rbrac{k^{-4/5}}} \frac{(k-j-b_c)!}{e}
 \end{align*}
 by Fact \ref{fact:hest}. So the product in \eqref{eqn:EYsquared} is
 \begin{align}
     \prod_{c=1}^\ell \binom{k-i}{b_c}\binom{k-b_c}{j} h(k-j-b_c, k-i-b_c-2j) & = \rbrac{1+O\rbrac{k^{-4/5}}} \prod_{c=1}^\ell \frac{(k-i)_{b_c}}{b_c!}\frac{(k-b_c)_j}{j!} \frac{(k-j-b_c)!}{e} \nonumber \\
     & \le  \rbrac{1+O\rbrac{k^{-4/5}}}\rbrac{\frac{k!}{ej!}}^\l \prod_{c=1}^\ell \frac{1}{b_c!}.\label{eqn:prodest1}
 \end{align}

 The sum of terms in \eqref{eqn:EYsquared} corresponding to small $i, j, b$ is at most 
  \begin{align}
     &\rbrac{1+O\rbrac{k^{-4/5}}} \sum_{\substack{0 \le i, j, b \le k^{1/10} \\ \bh \in S_b}} p^{2k\ell-(i+j)\ell-b}\exp\left\{2\ell\left(1-\frac{1}{p}\right) \right\}. \left(\frac{k!}{i!}\right)^\ell \rbrac{\frac{k!}{ej!}}^\l \prod_{c=1}^\ell \frac{1}{b_c!}\nonumber \\
     & =\rbrac{1+O\rbrac{k^{-4/5}}} \exp\left\{\ell\left(1-\frac{2}{p}\right) \right\}\sum_{0 \le i, j, b \le k^{1/10}} p^{2k\ell-(i+j)\ell-b} \left(\frac{k!}{i!}\right)^\ell \rbrac{\frac{k!}{j!}}^\l \frac{\l^b}{b!}\nonumber \\
     & \le \rbrac{1+O\rbrac{k^{-4/5}}} (k!)^{2\l} p^{2k\l}\exp\left\{\ell\left(1-\frac{1}{p}\right) \right\}\sum_{0 \le i, j , b} \frac{p^{-i\ell} }{(i!)^\l} \cdot  \frac{p^{-j\ell} }{(j!)^\l} \cdot  \frac{(\l/p)^b}{b!} \nonumber\\
     &=  \rbrac{1+O\rbrac{k^{-4/5}}} (k!)^{2\l} p^{2k\l}\exp\left\{\ell\left(1-\frac{1}{p}\right) \right\} f_\l\rbrac{\frac1p}^2 \nonumber\\
     &\sim \E[Y]^2 \nonumber
 \end{align}
 where on the second-to-last line we have used 
 \[
 \sum_{0 \le i} \frac{p^{-i\ell} }{(i!)^\l} = f_\l\rbrac{\frac1p}, \qquad \qquad \sum_{0 \le b }  \frac{(\l/p)^b}{b!} =  \exp\cbrac{\frac {\l}{p}}.
 \]
 
 It remains to show that the sum of all other terms (i.e. terms where $i, j,$ or $b$ is at least $k^{1/10}$) is negligible compared to $\E[Y]^2$, which is of order $(k!)^{2\l} p^{2k\l}$. Note that by Fact
\ref{fact:edgeprob} \begin{align*}
     \frac{\binom{k^2\ell-(2k\ell-(i+j)\ell-b)}{m-(2k\ell-(i+j)\ell-b)} }{\binom{k^2\ell}{m} } &= p^{2k\ell-(i+j)\ell-b}\exp\left\{\frac{(2k\ell-(i+j)\ell-b)^2}{2k^2\l}\rbrac{1-\frac{1}{p}} +O\rbrac{\frac{1}{k}}\right\} \\
     &= O\rbrac{p^{2k\ell-(i+j)\ell-b}}.
 \end{align*}
 Thus, the sum (over $\bh$) of terms corresponding to a fixed triple $(i, j, b)$ in line \eqref{eqn:EYsquared} big-O of  
 \begin{align*}
      & p^{2k\ell-(i+j)\ell-b}  \sum_{\bh \in S_b} \left(\frac{k!}{i!}\right)^\ell \prod_{c=1}^\ell \binom{k-i}{b_c}\binom{k-b_c}{j} (k-j-b_c)! \nonumber\\
      & \le p^{2k\ell-(i+j)\ell-b} \left(\frac{k!}{i!}\right)^\ell \left(\frac{k!}{j!}\right)^\ell \sum_{\bh \in S_b}  \prod_{c=1}^\ell \frac{1}{b_c!} \nonumber\\
      & =\rbrac{(k!)^{2\l} p^{2k\l} } \cdot \rbrac{ \frac{\l^b}{p^{(i+j)\ell+b}(i!)^\ell(j!)^\l b!} }.
 \end{align*}
It is easy to see that if $i, j$ or $b$ is at least $k^{1/10}$ then the second factor above is $\exp\cbrac{-\Omega\rbrac{k^{1/10} \log k}}$. Since the number of triples $(i, j, b)$ is polynomial in $k$, the sum of all such terms (i.e. where $i, j$ or $b$ is at least $k^{1/10}$) is $o\rbrac{(k!)^{2\l} p^{2k\l} }$ which is a negligible contribution to $\E[Y^2]$. Therefore we have $\E[Y^2]\sim \E[Y]^2.$

\section{Remarks and Open Problems}

The reader should note that we did not use a ``binomial" random construction (e.g. keep each edge of $D_{k, \ell}$ with probability $p$ independently) because such a model lacks the concentration we need here. Indeed, for example Janson (\cite{J94}) showed that the number of perfect matchings in $G(n, p)$ is not concentrated even when it is quite large, while the number of perfect matchings of $G(n, m)$ is concentrated. We tried to use a binomial random construction and found that the second moments were too large, which in light of Janson's result makes sense (for example derangements in our graph are just a union of several perfect matchings on bipartite graphs). 

There are still interesting open problems in \cite{BDHHL}. In particular it is still open whether $S$, the set of possible ratios $(d/p)_G$, is equal to $\mathbb{Q} \cap [0, 1/2]$. Here we would like to pose another open problem that is mostly unrelated to our result. In particular, we ask about stability for digraphs whose ratio $(d/p)_G$ is close to $1/2$: do such digraphs have to resemble the blow-up of a directed graph?

\bibliographystyle{abbrv}
\bibliography{refs}

\begin{thebibliography}{1}

\bibitem{BDHHL}
M.~Bucic, P.~Devlin, M.~Hendon, D.~Horne, and B.~Lund.
\newblock Perfect matchings and derangements on graphs, 2019.

\bibitem{FK}
A.~Frieze and M.~Karo\'{n}ski.
\newblock {\em Introduction to random graphs}.
\newblock Cambridge University Press, Cambridge, 2016.

\bibitem{J94}
S.~Janson.
\newblock The numbers of spanning trees, {H}amilton cycles and perfect
  matchings in a random graph.
\newblock {\em Combin. Probab. Comput.}, 3(1):97--126, 1994.

\end{thebibliography}

\end{document}